\documentclass{article}

\usepackage{amssymb}
\usepackage{amsthm}
\usepackage{amsmath}

\usepackage{a4wide}

\usepackage{graphicx}
\usepackage{float}

\usepackage{cite}
\usepackage{url}

\newtheorem{theorem}{Theorem}
\newtheorem{proposition}[theorem]{Proposition}
\newtheorem{corollary}[theorem]{Corollary}

\theoremstyle{definition}
\newtheorem{definition}[theorem]{Definition}

\def\cprime{$'$}

\begin{document}

\title{Control of a Novel Chaotic Fractional Order System\\
Using a State Feedback Technique\thanks{This is a preprint 
of a paper whose final and definite form will appear in
\emph{Mechatronics}. Paper submitted 06-Dec-2012; revised 25-Feb-2013 
and 02-March-2013; accepted for publication 10-March-2013.}}

\author{Abolhassan Razminia$^1$\\
\texttt{a.razminia@gmail.com}
\and Delfim F. M. Torres$^2$\\
\texttt{delfim@ua.pt}}


\date{$^1$Department of Electrical Engineering, School of Engineering,\\
Persian Gulf University, Bushehr, Iran\\[0.3cm]
$^2$CIDMA--Center for Research and Development in Mathematics and Applications,\\
Department of Mathematics, University of Aveiro, 3810-193 Aveiro, Portugal}

\maketitle


\begin{abstract}
We consider a new fractional order chaotic system
displaying an interesting behavior.
A necessary condition for the system to remain chaotic is derived.
It is found that chaos exists in the system with order less than three.
Using the Routh--Hurwitz and the Matignon stability criteria,
we analyze the novel chaotic fractional order system
and propose a control methodology that is better than
the nonlinear counterparts available in the literature,
in the sense of simplicity of implementation and analysis.
A scalar control input that excites only one of the states is proposed,
and sufficient conditions for the controller gain to stabilize
the unstable equilibrium points derived.
Numerical simulations confirm the theoretical analysis.

\medskip

\noindent \textbf{Keywords:} chaos; control of chaos; fractional order systems; stability tests.
\end{abstract}


\section{Introduction}

Fractional calculus has been known since the early 17th century \cite{[01]}.
It has been extensively applied in many fields, with
an overwhelming growth of applications during the last three decades.
Examples abound in physics \cite{[02]}, engineering \cite{[03]},
mathematical biology \cite{[04]}, finance \cite{[05]},
life sciences \cite{[06]}, and optimal control \cite{[07]}.
This is due to the fact that, in many applications,
approaches based on fractional derivatives establish
far superior models of engineering systems than the
approaches based on ordinary derivatives \cite{Zhou:Kuang}.
As mentioned in \cite{[01]}, there is no field
that has remained untouched by fractional derivatives.

Historically, the lack of a physical interpretation of fractional derivatives
has been acknowledged at the first international conference on the fractional
calculus in New Haven (USA), in 1974, by including it in the list of open problems
\cite{Ref1:[1]}. The question was not answered, and therefore repeated at the
subsequent conferences at the University of Strathclyde (UK) in 1984 \cite{Ref1:[2]}
and at the Nihon University (Tokyo, Japan) in 1989 \cite{Ref1:[3]}. The round-table
discussion at the conference on transform methods and special functions in Varna (1996)
showed that the problem was still unsolved \cite{Ref1:[4]}. Since then,
the geometric and physical interpretation of fractional derivatives
has been studied by several authors \cite{Ref1:[7],[08]}.
An interesting physical discussion about initial conditions of fractional order systems
is reported in \cite{[09]}, and their role in control theory is addressed
in \cite{MyID:163} for Caputo fractional derivatives, and
in \cite{MyID:181} for Riemann--Liouville fractional derivatives.
Relation between fractional integrals and derivatives
and fractal geometry, showing that for some complex systems substitution
of integer-order derivatives by fractional ones results in more accurate and superior models,
has received special attention \cite{Ref1:[5],Ref1:[6]}.

Chaos is an interesting phenomenon in nonlinear dynamical systems
that has been developed and thoroughly studied over the past two decades.
The reader interested in applications of chaos in medicine and biology,
where fractional calculus has initiated its success and activity in engineering applications,
is referred to \cite{MR2595930,MR1103340}.
A chaotic system is a nonlinear deterministic system that displays complex,
noisy-like and unpredictable behavior. The sensitive dependence on the initial conditions,
and the system parameter variation, is a prominent characteristic of chaotic behavior.
Here we consider a fractional order chaotic system. The corresponding integer-order system
has a chaotic behavior for a wide range of parameters. Such integer-order dynamical system
is a reduced model for a physical system. When such system is implemented using physical
electronic devices, the environment effects (aging of the elements, temperature, inaccurate values,
and so on) on the elements appear as a different behavior so that the response predicted
by the model does not resemble the actual system. Indeed, for such chaotic systems
the super sensitivity of the values of the elements to tiny changes, cannot be considered
in the dynamical equations. The main reason is related to the fractal or holed basin
of the invariant set of the system. In other words, some trace of the system trajectories
are seen in the observed coordinate (phase plane) and some of them lie in the unseen region
that cannot be handled by the classical nonlinear ODE model. For this reason,
the fractional order system is more appropriate than the classical one
\cite{Ref1:[1],Ref1:[2],Ref1:[3],Ref1:[4],Ref1:[5],Ref1:[6],Ref1:[7]}.
This happens to be a frequent situation: for most physical systems which exhibit chaotic behavior, the invariant set
is not an integer-order dimensional object, and the basin of the trajectories in the phase space
is a strange attractor field whose Lyapunov dimensions are non-integer.
The fractional order operators allow to describe the observed behavior by an appropriate kernel in the integral.
This kernel can be treated as a weighting factor that generates a new response using existing vector fields.

The study of fractional order dynamical systems has attracted
an increasing attention in recent years due to their great promise
as a valuable tool in the modeling of many phenomena \cite{[10]}.
As a matter of fact, real world processes generally or most likely
are fractional order systems \cite{[11]}. It has been found that fractional order
systems possess memory and display more sophisticated dynamics when compared
to their integer order counterparts, which is of great significance in secure communications.
On the other hand, due to their potential applications in laser physics,
chemical reactors, secure communications and economics, a new direction
of chaos research has emerged in the past two decades to address the more
challenging problem of chaos synchronization and control \cite{[12],[13],[14],[15]}.
Recent papers study the control and synchronization of chaotic systems
in both integer and fractional order cases, applying various control methodologies:
nonlinear control \cite{[16],Wang:Jia}, adaptive control \cite{[17],Motallebzadeh}, robust control \cite{[18]},
fuzzy control \cite{[19]}, and active control \cite{Razminia}.
The main problem in applying such control methodologies
for taming chaos is often their complexity in implementation. In this paper we propose
a simple linear control mechanism, based on the well-known Routh--Hurwitz stability criterion,
for a fractional order dynamical system.
The system shows very rich nonlinear dynamics,
including chaos and period doubling bifurcations.
A controller is proposed, and both analysis and design are studied.
We show that a single input can control the very complex system.
This makes our results simpler than the conventional nonlinear methods available in the literature
and more feasible to implement.

The rest of the manuscript is organized as follows.
Section~\ref{sec:2} briefly presents the necessary fractional calculus background.
In Section~\ref{sec:3} we recall the stability criteria that are used
in our subsequent analysis and design. Description of an interesting system
is presented in Section~\ref{sec:4}, while the new results are given in Section~\ref{sec:5},
where we illustrate our control methodology for taming the fractional order chaotic system
corresponding to the one of Section~\ref{sec:4}.
Conclusions and future directions of research are given in Section~\ref{sec:6}.


\section{Fractional Calculus Background}
\label{sec:2}

In this section some necessary mathematical background is presented.
For more details see the books
\cite{book:Magin,book:Mal:Tor,book:Oldham,book:Oustaloup,book:Podlubny}.

\begin{definition}[see, \textrm{e.g.}, \cite{[20]}]
The (left) fractional integral of $x$ of order $q$,
$q \in \mathbb{R}^+$, is defined by
\begin{equation*}
{}_aD_t^{ - q}x(t): = \frac{1}{\Gamma (q)}
\int_a^t {{{(t - s)}^{q - 1}}x(s)\,ds},
\end{equation*}
where $\displaystyle \Gamma (q) = \int_0^\infty  {{e^{ - z}}{z^{q - 1}}dz}$
is the Gamma function.
\end{definition}

\begin{definition}[see, \textrm{e.g.}, \cite{[16]}]
The (left) fractional derivative of $x$ of order $q$, $q \in \mathbb{R}^+$,
in the sense of Riemann--Liouville, is defined by
\begin{equation*}
{}_a^{RL}D_t^qx(t):= {D^m}{}_aD_t^{ - (m - q)}x(t)
= \frac{1}{\Gamma (m - q)}\frac{{{d^m}}}{d{t^m}}
\int_a^t {{{(t - s)}^{m - q - 1}}x(s)ds},
\end{equation*}
where $m \in {{\mathbb{Z}}^+}$ is such that $m - 1 < q < m$.
\end{definition}

\begin{theorem}[see, \textrm{e.g.}, \cite{[21]}]
\label{thm:rem}
For the fractional Riemann--Liouville derivative and integral one has:
\begin{enumerate}
\item $\mathcal{L}\left\{{{}_aD_t^{ - q}x(t)}\right\} = {s^{ - q}}X(s)$;

\item $\mathop{\lim}\limits_{q \to m} {}_0D_t^{ - q}x(t) = D^{-m}x(t)$,
$q > 0$, $m \in {{\mathbb{Z}}^+}$;

\item \label{item:3} ${}_0^{RL}D_t^qc = {\frac{c{t^{q - 1}}}{\Gamma (1 - q)}}$;

\item ${}_0^{RL}D_t^q{}_0D_t^{-q}x(t) = x(t)$, $q \in {{\mathbb{R}}^+}$;

\item \label{item:5} $\mathcal{L}\left\{ {{}_0^{RL}D_t^qx(t)} \right\}
= {s^q}X(s) - \sum\limits_{k = 0}^{m - 1} {{s^k} \cdot {}_0^{RL}D_t^{q - k - 1}x(0)}$,
$m - 1 < q < m$, $m \in {{\mathbb{Z}}^ +}$;
\end{enumerate}
where $\mathcal{L}$ denotes the Laplace transform, $c$ a constant,
and $D^{-m}$ the $m$-folded integral.
\end{theorem}

We note that the Riemann--Liouville differentiation of a constant is not zero
(item~\ref{item:3} of Theorem~\ref{thm:rem});
while its Laplace transform needs fractional derivatives
of the function in initial time (item~\ref{item:5} of Theorem~\ref{thm:rem}).
To overcome these ``imperfections'', the Caputo fractional derivative
has been introduced.

\begin{definition}[see, \textrm{e.g.}, \cite{[16]}]
The (left) fractional derivative of $x$ of order $q$, $q \in \mathbb{R}^+$,
in the sense of Caputo, is defined by
\begin{equation*}
{}_a^CD_t^qx(t): = {}_a^{RL}D_t^{ - (m - q)}{D^m}x(t)
= {\frac{1}{\Gamma (m - q)}}\int_a^t {{{(t - s)}^{m - q - 1}}{x^{(m)}}(s)ds},
\end{equation*}
where $m \in {{\mathbb{Z}}^+}$ is such that $m - 1 < q < m$,
$D^m$ is the standard differential operator of order $m$.
\end{definition}

\begin{theorem}[see, \textrm{e.g.}, \cite{[17]}]
For the fractional Caputo derivative one has:
\begin{enumerate}
\item ${}_0^CD_t^qc = 0$;

\item ${}_0^CD_t^q{}_0D_t^{ - q}x(t)
= {}_0^{RL}D_t^q{}_0D_t^{ - q}x(t) = x(t)$, $0 < q < 1$;

\item $\mathcal{L}\left\{ {{}_0^CD_t^qx(t)} \right\}
= {s^q}X(s) - \sum\limits_{k = 0}^{m - 1} {{s^{q - k - 1}}{x^{(k)}}(0)}$,
$m - 1 < q < m$, $m \in {{\mathbb{Z}}^+}$;
\end{enumerate}
where $c$ denotes a constant and $\mathcal{L}$ the Laplace transform.
\end{theorem}

In the next section we study some stability tests for fractional order systems.


\section{Stability Criteria}
\label{sec:3}

A fractional order dynamical system is usually described by

\begin{equation}
\label{eq:4}
\begin{cases}
{{}_0^CD_t^qx(t) = f(x(t),t),\,\,\,\,\,m - 1 < q < m,\  m \in {{\mathbb{Z}}^+},\  t > 0},\\
{\left[ {{}_0^{RL}D_t^kx(t)} \right]\left| {_{t = 0}} \right.
= x_0^k,\qquad k = 0, \ldots ,\,m - 1},
\end{cases}
\end{equation}
where $x(t) \in {{\mathbb{R}}^n}$ is the vector state at time $t$,
$f:{{\mathbb{R}}^n} \times {\mathbb{R}} \to {{\mathbb{R}}^n}$
the nonlinear vector field, and
$q = \left({{q_1}}, \ldots, {{q_n}}\right)^T$
the differentiation order vector.
If ${q_1} = {q_2} =  \cdots  = {q_n} =: \alpha$,
we call \eqref{eq:4} a commensurate fractional order dynamical system;
otherwise, \eqref{eq:4} is said to be incommensurate. The sum of the orders
of all involved derivatives in Eq.~\eqref{eq:4}, \textrm{i.e.},
$\sum\limits_{i = 1}^n {{q_i}}$, is called the effective dimension
of Eq.~\eqref{eq:4} \cite{[22]}. The size of vector $x(t)$ in state space form \eqref{eq:4},
\textrm{i.e.}, $n$, is called the inner dimension of system \eqref{eq:4} \cite{[23]}.

\begin{theorem}[see \cite{[19]}]
The commensurate order system
\begin{equation}
\label{eq:5}
{}_0^CD_t^\alpha x(t) = Ax(t), \quad x(0) = {x_0},
\end{equation}
with $0 < \alpha \le 1$, $x(t) \in {{\mathbb{R}}^n}$,
and $A \in {{\mathbb{R}}^{n \times n}}$,
is asymptotically stable if, and only if,
$\left| {\arg \,(\lambda )} \right| > \alpha{\frac{\pi}{2}}$
for all eigenvalues $\lambda$ of $A$.
Moreover, system \eqref{eq:5} is stable if, and only if,
$\left| {\arg \,(\lambda )} \right| \ge \alpha{\frac{\pi}{2}}$
for all eigenvalues $\lambda$ of $A$, with those critical
eigenvalues satisfying $\left| {\arg \,(\lambda )} \right|
= \alpha{\frac{\pi}{2}}$ having geometric multiplicity of one.
\end{theorem}

\begin{theorem}[see \cite{[24]}]
Consider the following linear fractional order system:
\begin{equation}
\label{eq:6}
{}_0^CD_t^qx(t) = Ax(t), \quad x(0) = {x_0},
\end{equation}
where $x(t) \in {{\mathbb{R}}^n}$, $A \in {{\mathbb{R}}^{n \times n}}$,
and $q = \left({{q_1}}, \ldots,  {{q_n}}\right)^T$, $0 < {q_i} \le 1$,
${q_i} = {\frac{{n_i}}{{d_i}}}$ and $\gcd ({n_i},{d_i}) = 1$,
$i = 1, \ldots, n$. If $M$ is the least common multiple of the denominators $d_i$,
$i = 1, \ldots, n$, then the zero solution of \eqref{eq:6} is globally asymptotically
stable in the Lyapunov sense if all roots $\lambda$ of equation
\begin{equation}
\label{eq:7}
\Delta(\lambda ) = \det \,\left( {diag\,({\lambda ^{M{q_i}}}) - A} \right) = 0
\end{equation}
satisfy $\left| {\arg \,(\lambda )} \right| > \frac{\pi}{2M}$.
\end{theorem}

\begin{theorem}[see \cite{[25]}]
\label{thm3}
Let $Q = \left({x^*},{y^*},{z^*}\right)$
be an equilibrium solution of \eqref{eq:4} when $n = 3$
and $0 < {q_1} = {q_2} = {q_3} =: \alpha \le 1$;
and the eigenvalues of the equilibrium point $Q$ for the Jacobian matrix
$J: = {\frac{\partial f}{\partial x}}\left|{_Q} \right.$
be given by the polynomial
$\Delta(\lambda ) = {\lambda ^3} + {a_1}{\lambda ^2} + {a_2}\lambda  + {a_3} = 0$
with discriminant
\begin{equation}
\label{eq:9}
D(\Delta) = 18{a_1}{a_2}{a_3} + {({a_1}{a_2})^2}
- 4{a_3}{({a_1})^3} - 4{({a_2})^3} - 27{({a_3})^2}.
\end{equation}
The following holds:
\begin{description}
\item[(i)] If $D(\Delta) > 0$, then a necessary and sufficient condition
for the equilibrium point $Q$ to be locally asymptotically stable is
${a_1} > 0$, ${a_3} > 0$, ${a_1}{a_2} - {a_3} > 0$.

\item[(ii)] If $D(\Delta ) < 0$ and ${a_1} \ge 0$, ${a_2} \ge 0$, ${a_3} \ge 0$,
then $Q$ is locally asymptotically stable for $\alpha < {\frac{2}{3}}$. However,
if $D(\Delta) < 0$, ${a_1} < 0$, ${a_2} < 0$, $\alpha > {\frac{2}{3}}$,
then all roots of Eq.~\eqref{eq:9} satisfy the condition
$\left| {\arg (\lambda )} \right| < \alpha{\frac{\pi}{2}}$.

\item[(iii)] If $D(\Delta ) < 0$, ${a_1} > 0$, ${a_2} > 0$, ${a_1}{a_2} - {a_3} = 0$,
then $Q$ is locally asymptotically stable for all  $\alpha \in \left({0,1}\right)$.

\item[(iv)] A necessary condition for the equilibrium point $Q$
to be locally asymptotically stable is ${a_3} > 0$.

\item[(v)] If the conditions $D(\Delta) < 0$, ${a_1} > 0$, ${a_2} > 0$,
${a_1}{a_2} - {a_3} = 0$ are satisfied, then the equilibrium point $Q$
is not locally asymptotically stable for $\alpha = 1$.
\end{description}
\end{theorem}


\section{An Interesting System}
\label{sec:4}

In \cite{[26]} the following
three-dimensional smooth system
is proposed and investigated:
\begin{equation}
\label{eq:10}
\left(
\begin{matrix}
   {\dot x}\\
   {\dot y}\\
   {\dot z}
\end{matrix}\right)
= \left(
\begin{matrix}
{y - ax + byz}\\
{cy - xz + z}\\
{dxy - hz}
\end{matrix}\right),
\end{equation}
where ${\left[ {x(t),\,y(t),\,z(t)} \right]^T} \in {\mathbb{R}^3}$
is the state vector, and $a$, $b$, $c$, $d$, and $h$ are some positive constants.
The system was shown to be chaotic in a wide parameter range,
and to have an interesting complex dynamical behavior
that varies according with the values of the parameters $a$, $b$, $c$, $d$, and $h$.
The very rich nonlinear dynamics include chaos
and period doubling bifurcations. In particular, the system generates
a two-scroll chaotic attractor for $(a,b,c,d,h)=(3, 2.7, 4.7, 2, 9)$.

Chaos may be seen in many real-life nonlinear systems. About ten years ago,
several experimental and theoretical studies have been done to depict
the chaotic behavior in various electronic systems: nonlinear circuits \cite{Ref2:[1]},
secure communications \cite{Ref2:[2],Ref2:[3]}, lasers \cite{Ref2:[4]},
and Colpitts oscillators \cite{Ref2:[5]}. The system \eqref{eq:10} is particularly
relevant in mechatronics, where it can be used as a carrier producer.
Indeed, because of its wide range chaoticity, one of its important applications
is in secure communication systems. In such a system, a chaotic carrier is used
to transmit the message signal over a channel.
The main motivation for employing such carriers is:
(i) complexity of the carrier, which increases the security of the modulated signal;
(ii) inherent orthonormality of the chaotic signals, avoiding the necessity to use
in-phase local oscillators as often done in telecommunication systems
in order to demodulate the original signal at the receiver;
(iii) wide band signal, which permits the carrier to transmit a wide band message over a noisy channel.
Motivated by the interesting behavior of \eqref{eq:10} obtained in \cite{[26]},
our main goal is to investigate the chaotic dynamics
of the corresponding fractional system.
This is done in the next section.


\section{Main Results}
\label{sec:5}

Consider a 3D autonomous fractional system
\begin{equation}
\label{eq:17}
{}_0^CD_t^{{q}}x(t) = f\left(x(t)\right),
\end{equation}
where $q = {\left( {{q_1},{q_2},{q_3}} \right)^T}$
is the fractional order of differentiation,
$x(t) \in {{\mathbb{R}}^3}$ is the state vector,
and $f:{{\mathbb{R}}^3} \to {{\mathbb{R}}^3}$ is
the nonlinear vector field. Let $Q = (x_1^*,x_2^*,x_3^*)$
be an equilibrium of the system \eqref{eq:17},
\textrm{i.e.}, let $f(Q) = 0$. We say that
$Q$ is a saddle point for \eqref{eq:17} if the eigenvalues
of the Jacobian matrix $J = {\frac{\partial f}{\partial x}}$
evaluated at $Q$ are $a$ and $b \pm j c$, where $ab < 0$ and
$c \ne 0$. A saddle point $Q$ is called a saddle point of index 1
if $a > 0$ and $b < 0$, and it is called a saddle point of index 2
if $a < 0$ and $b > 0$. In chaotic systems of Shil\cprime nikov type,
scrolls in a chaotic attractor are generated only around the saddle points of index 2.
Moreover, saddle points of index 1 are responsible only for connecting scrolls \cite{[29]}.


\subsection{System Description}

We are interested in the particular case of \eqref{eq:17} that corresponds
to the commensurate fractional order version of \eqref{eq:10}.
For that we substitute the standard/integer order derivatives in \eqref{eq:10}
by Caputo fractional derivatives of order $\alpha \in (0,1)$:
\begin{equation}
\label{eq:16}
\left(
\begin{matrix}
{{}_0^CD_t^{{\alpha}}x}\\
{{}_0^CD_t^{{\alpha}}y}\\
{{}_0^CD_t^{{\alpha}}z}
\end{matrix}\right)
= \left(
\begin{matrix}
{y - ax + byz}\\
{cy - xz + z}\\
{dxy - hz}
\end{matrix}\right).
\end{equation}
System \eqref{eq:16} can be used to model several mechatronic systems.
One possible application is to model the nonlinear dynamics of a rotor-bearing system
with the purpose of diagnosing malfunctions and effectively improve
the dynamic characteristics of the rotor system.
Chu and Zhang analyzed the bifurcation and chaotic motion of a rub-impact rotor system.
They found three different routes to chaos with an increasing rotating speed \cite{Ref2:[6]}.
Later, Chu observed very rich forms of periodic and chaotic vibrations through experimental verification.
These results are of great importance to the fault diagnosis of the rub-impact problem \cite{Ref2:[7]}.
Ehrich studies the bifurcation of a bearing-rotor system, identifying a sub-harmonic vibration phenomenon
in a rotor dynamic system \cite{Ref2:[8]}. Goldman and Muszynska analyze the chaotic behavior
of a rub-impact rotor using numerical emulation and simple experimental verification.
They conclude that the rub can lead to higher order harmonics, sub-harmonic fractional frequencies,
or to chaotic vibrations \cite{Ref2:[9]}. Lin et al. analyze the nonlinear behavior
of rub-related vibration in rotating machinery: the effects of the rotating speed, clearance,
damping factors, friction coefficients, and boundary stiffness are investigated \cite{Ref2:[10]}.
Our system \eqref{eq:16} can also be regarded as a model for a DC-motor with chaotic behavior
(self-sustained oscillations according to backlash and dead-zone of the gears) \cite{DC:drive}.
A schematic diagram of a DC drive and its circuits, with separate excitation
and controller with hysteresis, can be found in \cite[Fig.~1]{DC:drive}.
The states of such system are the current $i_a$ in the motor armature circuit;
the current $i_f$ in the excitation coil;
and the rotor angular speed $\omega_r$.
This real life system can be described by \eqref{eq:16}, after normalization
and rescale according to the parameters and control signals of the circuit,
with the correspondence $x \leftrightarrow \omega_r$,
$y \leftrightarrow i_f$, and $z \leftrightarrow i_a$.

To find the equilibria of system \eqref{eq:16}, it is enough to equate
the right-hand side of \eqref{eq:16} to zero:
${y - ax + byz} = 0$,
${cy - xz + z} = 0$,
${dxy - hz} = 0$.
One concludes that the system has 5 equilibria:
\begin{equation}
\label{eq:12}
\begin{split}
{Q_1} &= \left(0, 0,  0\right),\\
{Q_2} &= \left(\frac{d + \sqrt \Delta}{2d},
\frac{h}{b}\left( \frac{- 1 + \sqrt {1 + \Lambda }}{d + \sqrt \Delta} \right),
\frac{ - 1 + \sqrt {1 + \Lambda }}{2b}\right),\\
{Q_3} &= \left(\frac{d + \sqrt \Delta}{2d},
\frac{h}{b}\left( \frac{ - 1 - \sqrt {1 + \Lambda }}{d + \sqrt \Delta} \right),
\frac{- 1 - \sqrt {1 + \Lambda }}{2b}\right),\\
{Q_4} &= \left( \frac{d - \sqrt \Delta}{2d},
\frac{h}{b}\left(\frac{- 1 + \sqrt {1 + \Lambda}}{d - \sqrt \Delta} \right),
\frac{- 1 + \sqrt {1 + \Gamma}}{2b}\right),\\
{Q_5} &= \left(\frac{d - \sqrt \Delta}{2d},
\frac{h}{b}\left(\frac{- 1 - \sqrt {1 + \Lambda }}{d - \sqrt \Delta}\right),
\frac{- 1 - \sqrt {1 + \Gamma}}{2b}\right),
\end{split}
\end{equation}
where
$\Delta = {d^2} + 4chd$,
$\Lambda = \frac{2ab}{h}\left( {d + 2ch + \sqrt \Delta  } \right)$,
and $\Gamma  = \frac{2ab}{h}\left( {d + 2ch - \sqrt \Delta  } \right)$.
The Jacobian matrix for \eqref{eq:16}, evaluated in an
equilibrium point $Q = \left({x^*},{y^*},{z^*}\right)$, is given by
\begin{equation*}
J = \left(
\begin{matrix}
{ - a} & {1 + b{z^*}} & {b{y^*}}\\
{ - {z^*}} & c & {1 - {x^*}}\\
{d{y^*}} & {d{x^*}} & { - h}
\end{matrix} \right).
\end{equation*}

Our first result gives a necessary condition on the fractional
order of differentiation $\alpha$, for chaos to occur.

\begin{theorem}[Necessary condition for occurrence
of a chaotic attractor in the fractional order system
\eqref{eq:16}]
\label{thm:new1}
If the fractional order system \eqref{eq:16} exhibits a chaotic attractor, then
\begin{equation}
\label{eq:20}
\alpha > \frac{2}{\pi}\arctan\,\left(
\frac{\left| {{\mathop{\rm Im}\nolimits}
\,(\lambda )} \right|}{{\mathop{\rm Re}\nolimits} \,(\lambda )}\right)
\end{equation}
for any eigenvalue $\lambda$ of $Q_i$ in \eqref{eq:12}, $i = 1, \ldots, 5$.
\end{theorem}

\begin{proof}
Assume that the 3D fractional system \eqref{eq:16} displays a chaotic attractor.
For every scroll existing in the chaotic attractor,
the system has a saddle point of index 2 encircled by its respective scroll.
Suppose that $\Omega$ is the set of equilibrium points of the system surrounded by scrolls.
A necessary condition for the fractional order system \eqref{eq:16}
to exhibit a chaotic attractor is instability of the equilibrium points in $\Omega$
\cite{[30]}. Otherwise, one of these equilibrium points becomes asymptotically stable
and attracts the nearby trajectories. According to \eqref{eq:7},
this necessary condition is mathematically equivalent to
\begin{equation}
\label{eq:18}
\frac{\pi}{2M} - \mathop {\min }\limits_i
\left\{ {\left| {\arg \,({\lambda _i})} \right|} \right\} \ge 0,
\end{equation}
where the $\lambda_i$ are the roots of
$\det\left( {diag\,\left( {{\lambda ^{M{q_1}}}\,\,{\lambda ^{M{q_2}}}\,\,
{\lambda ^{M{q_3}}}} \right) - J\left| {_Q} \right.} \right) = 0$
for all $Q \in \Omega$. Condition \eqref{eq:20} follows immediately from \eqref{eq:18}.
\end{proof}

The nature of the equilibria \eqref{eq:12} of \eqref{eq:16} may be determined
using the corresponding eigenvalues $\lambda$.
The following proposition lists the eigenvalues of each equilibrium,
when the parameters are selected in agreement with Section~\ref{sec:4}.

\begin{proposition}
Consider the fractional system \eqref{eq:16} of commensurate
order $\alpha \in (0,1)$, when the five parameters are selected
to be $(a,b,c,d,h)=(3, 2.7, 4.7, 2, 9)$. Then, the eigenvalues
$\lambda$ for each equilibrium $Q_i$ in \eqref{eq:12},
$i = 1, \ldots, 5$, are given as follows.
\begin{description}
\item[(i)] Eigenvalues of ${Q_1}$: $-9$, $- 3$ and $4.7$.

\item[(ii)] Eigenvalues of ${Q_2}$:
$- 11.0247$ and $1.8623 \pm j 6.6831$.

\item[(iii)] Eigenvalues of ${Q_3}$:
$- 11.7856$ and $2.2428 \pm j 6.8580$.

\item[(iv)] Eigenvalues of ${Q_4}$:
$- 10.7669$ and $1.7335 \pm j 6.0024$.

\item[(v)] Eigenvalues of ${Q_5}$:
$- 11.6813$ and $2.1906 \pm j 6.1881$.
\end{description}
\end{proposition}

\begin{proof}
Follows by direct computations.
\end{proof}

\begin{corollary}
For the fractional system \eqref{eq:16} with
$(a,b,c,d,h)=(3, 2.7, 4.7, 2, 9)$, the equilibria
${Q_2}$, ${Q_3}$, ${Q_4}$ and ${Q_5}$ are saddle points of index 2.
\end{corollary}

We conclude that if there are some chaotic attractors for $(a,b,c,d,h)=(3, 2.7, 4.7, 2, 9)$,
they are located around the equilibria ${Q_2}$, ${Q_3}$, ${Q_4}$, ${Q_5}$.
Examining \eqref{eq:20} for these equilibria, we obtain:
$\alpha > 0.8270$ for $Q_2$;
$\alpha > 0.7988$ for $Q_3$;
$\alpha > 0.8210$ for $Q_4$;
and $\alpha > 0.7834$ for $Q_5$.
Therefore, by choosing $\alpha > 0.8270$,
we ensure that all the eigenvalues remain in the instability region.

An efficient method for solving fractional order differential equations
is the predictor-corrector scheme or, more precisely,
the PECE (Predict, Evaluate, Correct, Evaluate)
technique that has been investigated in \cite{[27],[28]}.
It represents a generalization of the Adams--Bashforth--Moulton algorithm.
We use the PECE scheme throughout the paper for numerical simulations.

In Figures~\ref{fig2} to \ref{fig8},
the initial conditions were selected as
$\left({{x_0},{y_0},{z_0}}\right) = \left({5, - 2,1}\right)$,
and only the fractional order of differentiation $\alpha$ changes.
When $\alpha \rightarrow 1$, our numerical results are in agreement
with \cite{[26]}. The numerical simulation
of the chaotic attractor for $\alpha \rightarrow 1$
is depicted in Fig.~\ref{fig2}.

\begin{figure}
\begin{center}
\includegraphics[scale=0.60]{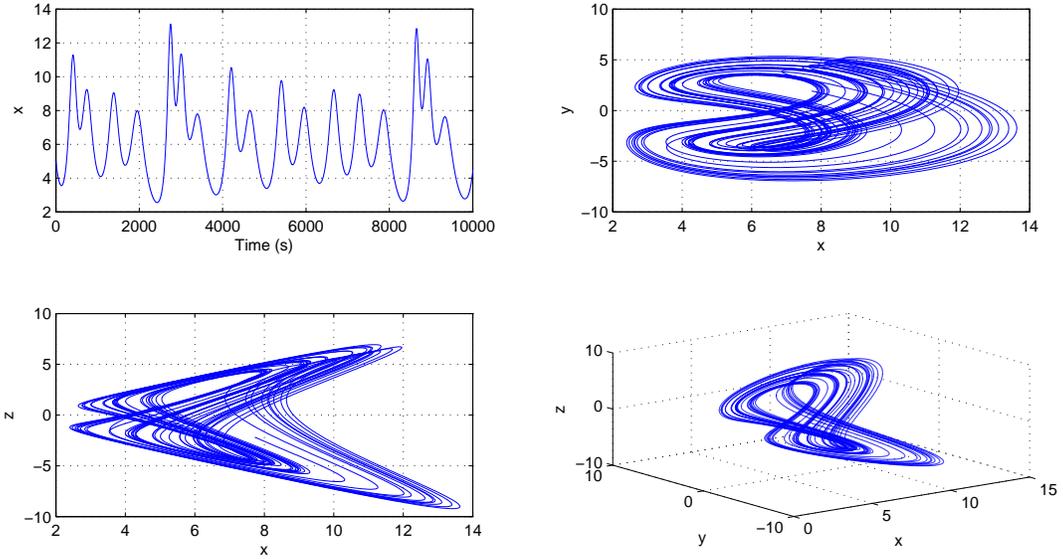}
\end{center}
\caption{Chaotic attractor of fractional system \eqref{eq:16}
with $a=3$, $b = 2.7$, $c = 4.7$, $d = 2$, $h = 9$,
and initial conditions $\left({{x_0},{y_0},{z_0}}\right) = (5, -2, 1)$,
when $\alpha \rightarrow 1$.}
\label{fig2}
\end{figure}

In Fig.~\ref{fig5} we illustrate the chaotic behavior of \eqref{eq:16}
when $\alpha = 0.90$, and in Fig.~\ref{fig6}, Fig.~\ref{fig7},
and Fig.~\ref{fig8}, the behavior of \eqref{eq:16} is depicted for $\alpha = 0.86$,
$\alpha = 0.80$, and $\alpha = 0.77$, respectively.

\begin{figure}
\begin{center}
\includegraphics[scale=0.60]{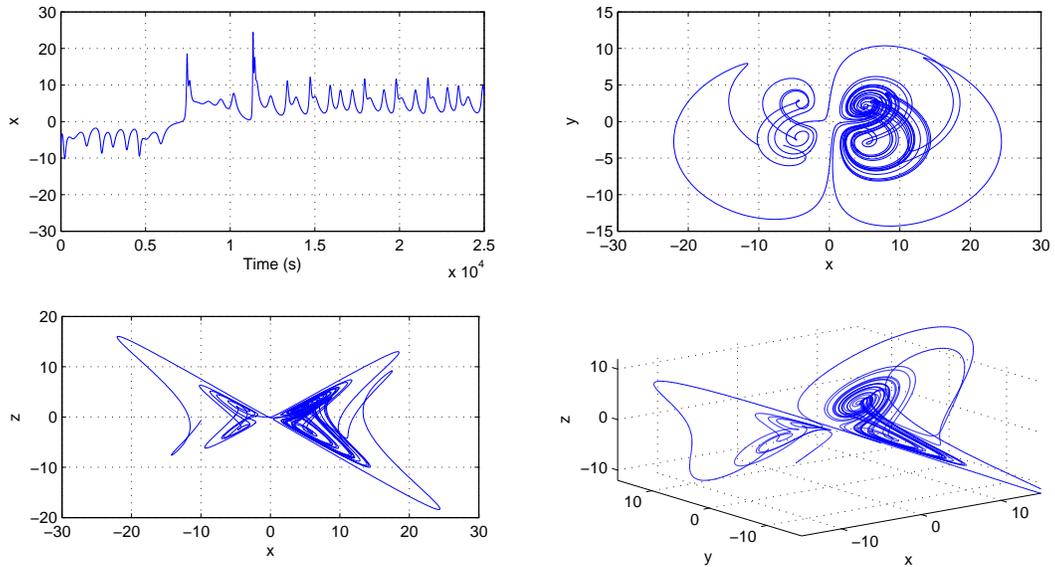}
\end{center}
\caption{Numerical results for the fractional order system \eqref{eq:16}
with $a=3$, $b = 2.7$, $c = 4.7$, $d = 2$, $h = 9$,
and initial conditions $\left({{x_0},{y_0},{z_0}}\right)=(5,-2,1)$,
when the fractional order is $\alpha = 0.90$.}
\label{fig5}
\end{figure}

\begin{figure}
\begin{center}
\includegraphics[scale=0.60]{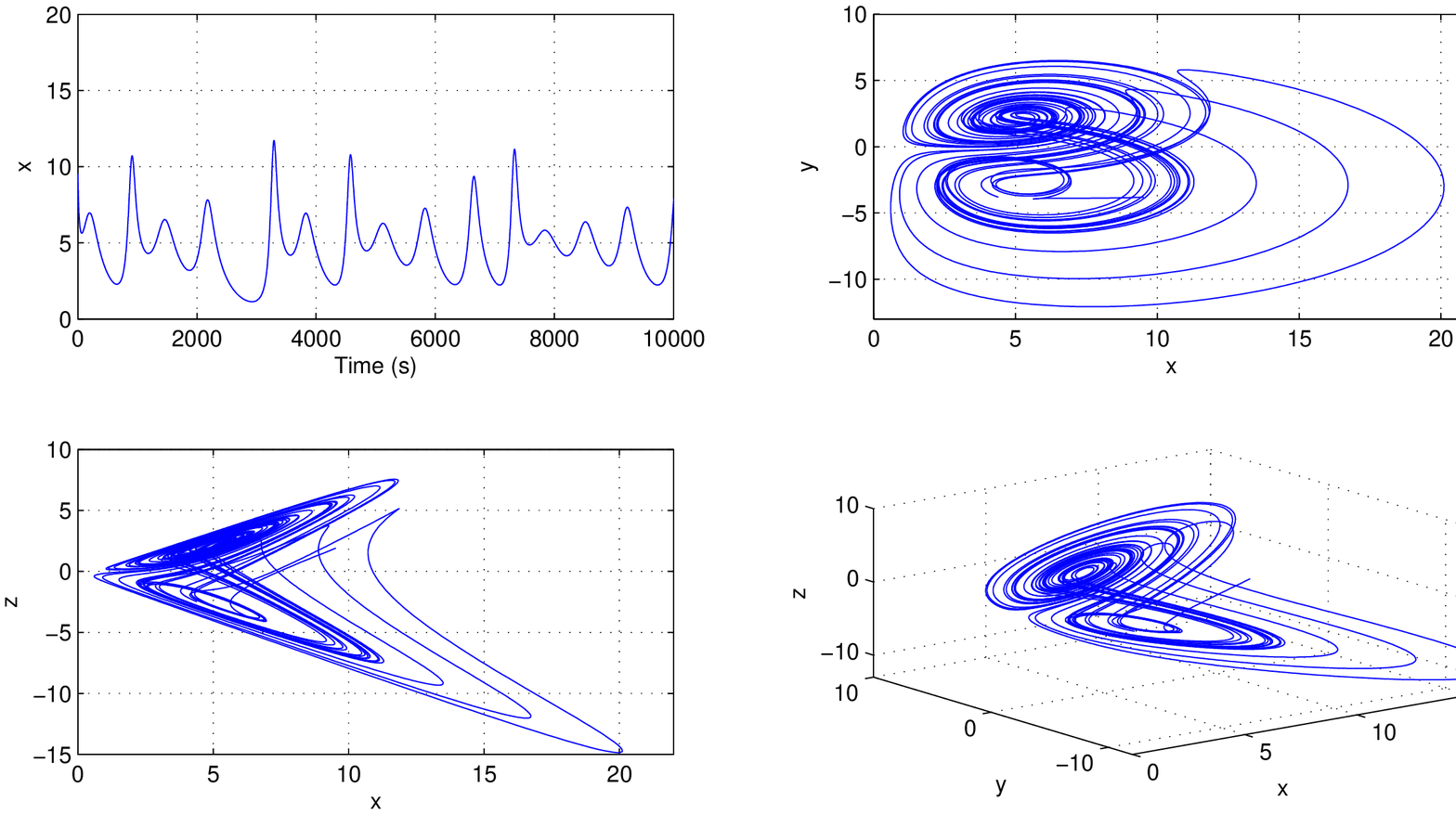}
\end{center}
\caption{Numerical results for the fractional order system \eqref{eq:16}
with $a=3$, $b = 2.7$, $c = 4.7$, $d = 2$, $h = 9$,
and initial conditions $\left({{x_0},{y_0},{z_0}}\right) = (5,-2,1)$,
when the fractional order is $\alpha = 0.86$.}
\label{fig6}
\end{figure}

\begin{figure}
\begin{center}
\includegraphics[scale=0.60]{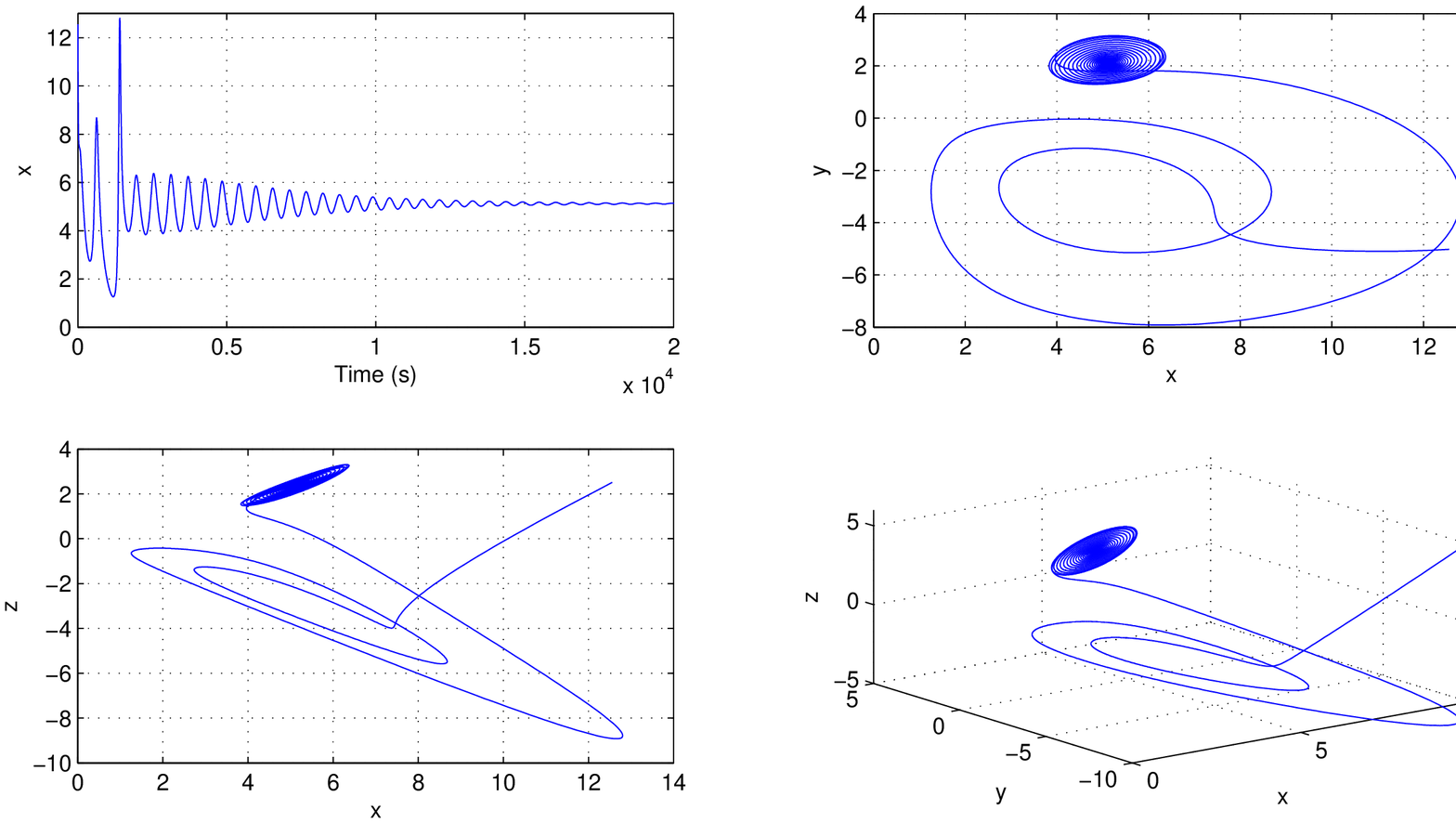}
\end{center}
\caption{Numerical results for the fractional order system \eqref{eq:16}
with $a=3$, $b = 2.7$, $c = 4.7$, $d = 2$, $h = 9$,
and initial conditions $\left({{x_0},{y_0},{z_0}}\right) = (5,-2,1)$,
when the fractional order is $\alpha = 0.80$.}
\label{fig7}
\end{figure}

\begin{figure}
\begin{center}
\includegraphics[scale=0.60]{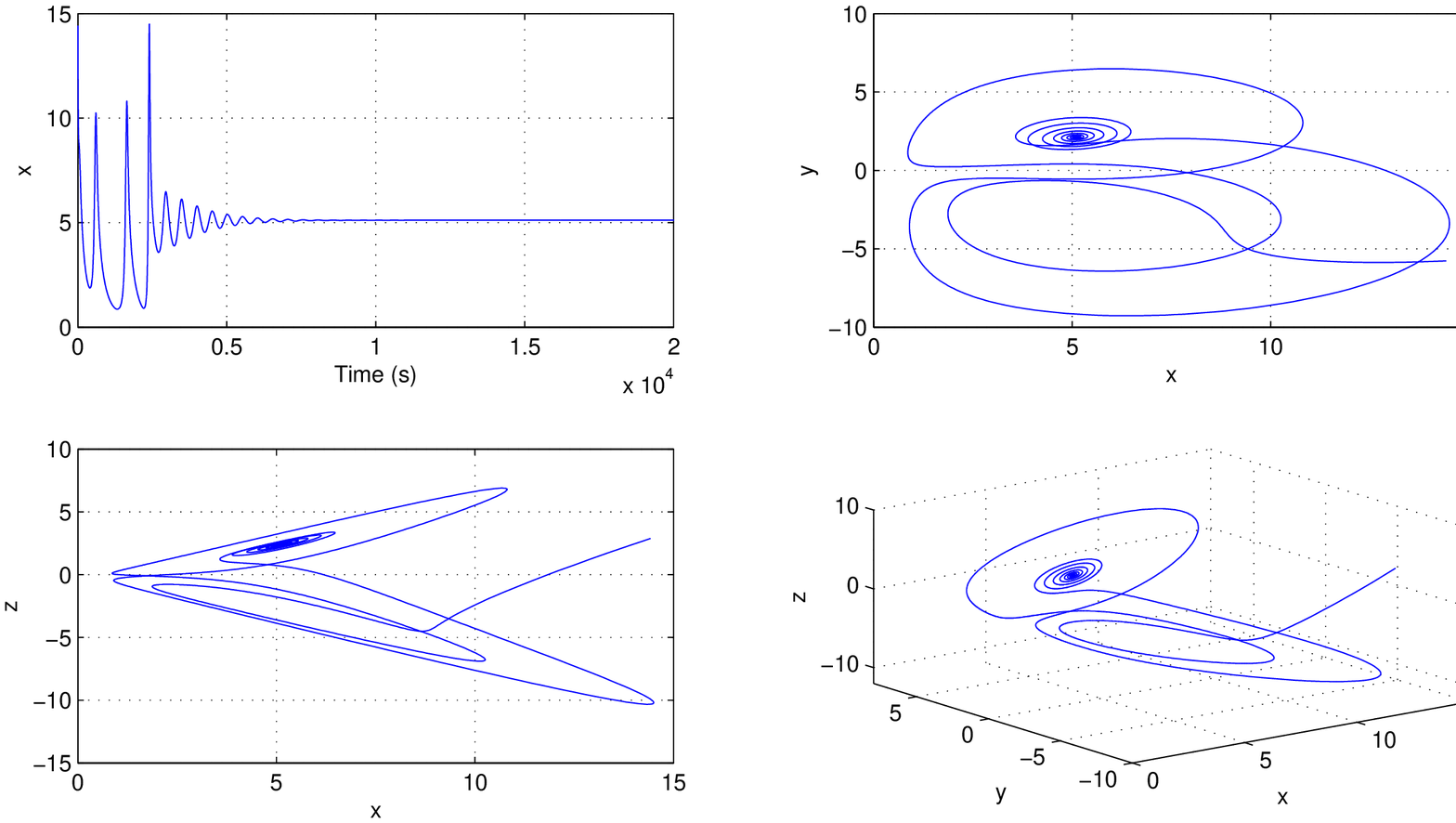}
\end{center}
\caption{Numerical results for the fractional order system \eqref{eq:16}
with $a=3$, $b = 2.7$, $c = 4.7$, $d = 2$, $h = 9$,
and initial conditions $\left({{x_0},{y_0},{z_0}}\right) = (5,-2,1)$,
when the fractional order is $\alpha = 0.77$.}
\label{fig8}
\end{figure}

As one can see, in Fig.~\ref{fig7}
and Fig.~\ref{fig8} chaos has diminished,
and the trajectories of the corresponding
fractional system converge to their equilibria.


\subsection{Control of the Fractional Order Chaotic System}

Consider the fractional order chaotic system \eqref{eq:16}
of commensurate order $\alpha \in (0, 1)$. In order to control the system,
\textrm{i.e.}, force the trajectories to go to the equilibria,
we add a control parameter $u = \left(u_1, u_2, u_3\right)$ as follows:
\begin{equation}
\label{eq:23}
\left(
\begin{matrix}
{{}_0^CD_t^{{\alpha}}x}\\
{{}_0^CD_t^{{\alpha}}y}\\
{{}_0^CD_t^{{\alpha}}z}
\end{matrix}\right)
= \left(
\begin{matrix}
{y - ax + byz}\\
{cy - xz + z}\\
{dxy - hz}
\end{matrix}\right)
+ \left(
\begin{matrix}
{{u_1}}\\
{{u_2}}\\
{{u_3}}
\end{matrix}\right).
\end{equation}
One of the simplest controllers is the state feedback controller,
which has a simple structure and is easy to implement.
Let us take the structure of the controller as a state feedback, as follows:
\begin{equation*}
\left\{
\begin{matrix}
{{u_1} =  - {k_1}(x - {x^*})},\\
{{u_2} =  - {k_2}(y - {y^*})},\\
{{u_3} =  - {k_3}(z - {z^*})}.
\end{matrix}\right.
\end{equation*}
In this way \eqref{eq:23} reduces to
\begin{equation}
\label{eq:25}
\left(
\begin{matrix}
{{}_0^CD_t^{{\alpha}}x}\\
{{}_0^CD_t^{{\alpha}}y}\\
{{}_0^CD_t^{{\alpha}}z}
\end{matrix}\right)
= \left(
\begin{matrix}
{y - ax + byz - {k_1}(x - {x^*})}\\
{cy - xz + z - {k_2}(y - {y^*})}\\
{dxy - hz - {k_3}(z - {z^*})}
\end{matrix}\right).
\end{equation}

Assume that we want to stabilize one of the equilibria, \textrm{e.g.}, ${Q_2}$
(using a similar approach, we can easily design a control law to stabilize
the other unstable equilibria). Next theorem shows
that by choosing appropriate values for gain ${k_1}$,
we can control the fractional order system \eqref{eq:25}.

\begin{theorem}
The trajectories of the fractional controlled system \eqref{eq:25}
with $a = 3$, $b = 2.7$, $c = 4.7$, $d = 2$, and $h = 9$,
are driven to the unstable equilibrium point
${Q_2}= \left( {5.1260,\,2.0794,\,2.3687} \right)$
for all $\alpha \in \left( {0,1} \right)$,
if ${k_2} = {k_3} = 0$ and $- 7.30 < {k_1} < 26.53$.
\end{theorem}

\begin{proof}
Computing the Jacobian matrix in the equilibrium point
$Q = \left({{x^*},{y^*},{z^*}}\right)$, we obtain
\begin{equation*}
J = \left(
\begin{matrix}
{ - a - {k_1}} & {1 + b{z^*}} & {b{y^*}}\\
{ - {z^*}} & {c - {k_2}} & { - {x^*} + 1}\\
{d{y^*}} & {d{x^*}} & { - h - {k_3}}
\end{matrix}\right).
\end{equation*}
Constituting the characteristic equation $\Delta(\lambda)$ by
\begin{equation*}
\Delta (\lambda ) = \det (\lambda I - J)
= \det \left(
\begin{matrix}
{\lambda  + a + {k_1}} & { - 1 - b{z^*}} & { - b{y^*}}\\
{{z^*}} & {\lambda  - c + {k_2}} & {{x^*} - 1}\\
{ - d{y^*}} & { - d{x^*}} & {\lambda  + h + {k_3}}
\end{matrix}\right) = 0,
\end{equation*}
we have
\begin{equation}
\label{eq:28}
\begin{split}
\Delta(\lambda ) &= \left( {\lambda  + a + {k_1}} \right)\left[
{\left( {\lambda  - c + {k_2}} \right)\left( {\lambda  + h
+ {k_3}} \right) + d{x^*}\left( {{x^*} - 1} \right)} \right]\\
& \quad + \left( {1 + b{z^*}} \right)\left[ {{z^*}\left( {\lambda
+ h + {k_3}} \right) + d{y^*}({x^*} - 1)} \right]
- b{y^*}\left[ {\left( { - d{x^*}{z^*}} \right)
+ d{y^*}\left( {\lambda  - c + {k_2}} \right)} \right] \\
&= {\lambda ^3} + \left( {a + {k_1} + {k_2} - c + h + {k_3}} \right){\lambda ^2}\\
& \quad + \left( {d{x^*}\left( {{x^*} - 1} \right) + (a + {k_1})({k_2}
- c + h + {k_3}) + ({k_2} - c)(h + {k_3}) + (1 + b{z^*}){z^*} - bd{y^*}^2} \right)\lambda\\
& \quad + (a + {k_1})(d{x^*}\left( {{x^*} - 1} \right) + ({k_2} - c)(h + {k_3}))
+ \left( {1 + b{z^*}} \right)(d{y^*}({x^*} - 1) + {z^*}(h + {k_3}))\\
& \quad - b{y^*}( - d{x^*}{z^*} + d{y^*}( - c + {k_2}))\\
&= 0.
\end{split}
\end{equation}
Based on Theorem~\ref{thm3}, if we choose ${k_1}$, ${k_2}$ and ${k_3}$
such that $D(\Delta ) < 0$, ${a_1} > 0$, ${a_2} > 0$ and ${a_1}{a_2} - {a_3} = 0$,
then $Q = \left({x^*},{y^*},{z^*}\right)$ is locally asymptotically stable
for all $\alpha \in \left( {0,1} \right)$, where
\begin{equation*}
\begin{split}
D(\Delta ) &= 18\left( {a + {k_1} + {k_2} - c + h + {k_3}} \right)\\
&\times \left( {d{x^*}\left( {{x^*} - 1} \right) + (a + {k_1})
({k_2} - c + h + {k_3}) + ({k_2} - c)(h + {k_3}) + (1 + b{z^*}){z^*} - bd{y^*}^2} \right)\\
&\times \Biggl( (a + {k_1})(d{x^*}\left( {{x^*} - 1} \right)
+ ({k_2} - c)(h + {k_3})) + \left( {1 + b{z^*}} \right)
\left(d{y^*}({x^*} - 1) + {z^*}(h + {k_3})\right)\\
&\qquad\qquad - b{y^*}\left( - d{x^*}{z^*} + d{y^*}( - c + {k_2})\right) \Biggr)\\
&+ \Biggl( \left( {a + {k_1} + {k_2} - c + h + {k_3}} \right)
\biggl( d{x^*}\left( {{x^*} - 1} \right) + (a + {k_1})({k_2} - c + h + {k_3})
+ ({k_2} - c)(h + {k_3})\\
&\qquad\qquad + (1 + b{z^*}){z^*} - bd{y^*}^2 \biggr) \Biggr)^2\\
&- 4\biggr( (a + {k_1})(d{x^*}\left( {{x^*} - 1} \right) + ({k_2} - c)(h + {k_3}))
+ \left( {1 + b{z^*}} \right)(d{y^*}({x^*} - 1) + {z^*}(h + {k_3}))\\
&\qquad\qquad - b{y^*}( - d{x^*}{z^*} + d{y^*}( - c + {k_2})) \biggr)
\left( {a + {k_1} + {k_2} - c + h + {k_3}} \right)^3\\
&- 4\biggl( d{x^*}\left( {{x^*} - 1} \right)
+ (a + {k_1})\left({k_2} - c + h + {k_3}\right)
+ ({k_2} - c)(h + {k_3}) + (1 + b{z^*}){z^*} - bd{y^*}^2 \biggr)^3\\
&- 27\biggl( (a + {k_1})\left(d{x^*}\left( {{x^*} - 1} \right) + ({k_2} - c)(h + {k_3})\right)
+ \left( {1 + b{z^*}} \right)(d{y^*}({x^*} - 1) + {z^*}(h + {k_3}))\\
&\qquad\qquad - b{y^*}\left(- d{x^*}{z^*} + d{y^*}( - c + {k_2})\right) \biggr)^3
\end{split}
\end{equation*}
and the ${a_i}$, $i = 1, 2, 3$, are found from the coefficients of \eqref{eq:28}.
For the parameters that generate the 2-scroll attractors, \textrm{i.e.},
$a = 3$, $b = 2.7$, $c = 4.7$, $d = 2$, $h = 9$, we have
${Q_2} = \left( {{x^*},{y^*},{z^*}} \right)
= \left( {5.1260,\,2.0794,\,2.3687} \right)$.
Thus, \eqref{eq:28} reduces to
\begin{multline*}
\Delta(\lambda) = {\lambda ^3} + \left( {7.3 + {k_1} + {k_2} + {k_3}} \right){\lambda ^2}
+ \left( {{k_1}{k_2} + {k_2}{k_3} + {k_1}{k_3} + 21.8733} \right)\lambda  + 530.6404\\
- 0.0002{k_1} + 3.6509{k_2} + 3.4177{k_3} + 3{k_2}{k_3} + 9{k_1}{k_2} - 4.7{k_1}{k_3} + {k_1}{k_2}{k_3}.
\end{multline*}
If we choose ${k_2} = {k_3} = 0$, then
\begin{equation*}
\Delta(\lambda) = {\lambda ^3} + \left( {7.3 + {k_1}} \right){\lambda ^2}
+ 21.8733\lambda  + 530.6404 - 0.0002{k_1}.
\end{equation*}
For this characteristic polynomial we have:
\begin{multline*}
D(\Delta) = 18\left( {7.3 + {k_1}} \right)\left( {21.8733} \right)\left(
{530.6404 - 0.0002{k_1}} \right)
+ {\left( {\left( {7.3 + {k_1}} \right)21.8733} \right)^2}\\
- 4\left( {530.6404 - 0.0002{k_1}} \right){\left( {7.3 + {k_1}} \right)^3}
- 4{\left( {21.8733} \right)^3} - 27{\left( {530.6404 - 0.0002{k_1}} \right)^2}.
\end{multline*}
The conditions $D(\Delta ) < 0$, ${a_1} > 0$, ${a_2} > 0$ and ${a_1}{a_2} - {a_3} = 0$
are satisfied for ${k_1} \in (-7.30, 26.53)$.
\end{proof}

In Fig.~\ref{fig9} we have chosen ${k_1} = 16.96$
and the initial conditions to be
$\left( {{x_0},{y_0},{z_0}} \right) = \left( {5,2,2} \right)$.
As can be seen from the figure, all states converge to their equilibria.

\begin{figure}
\begin{center}
\includegraphics[scale=0.60]{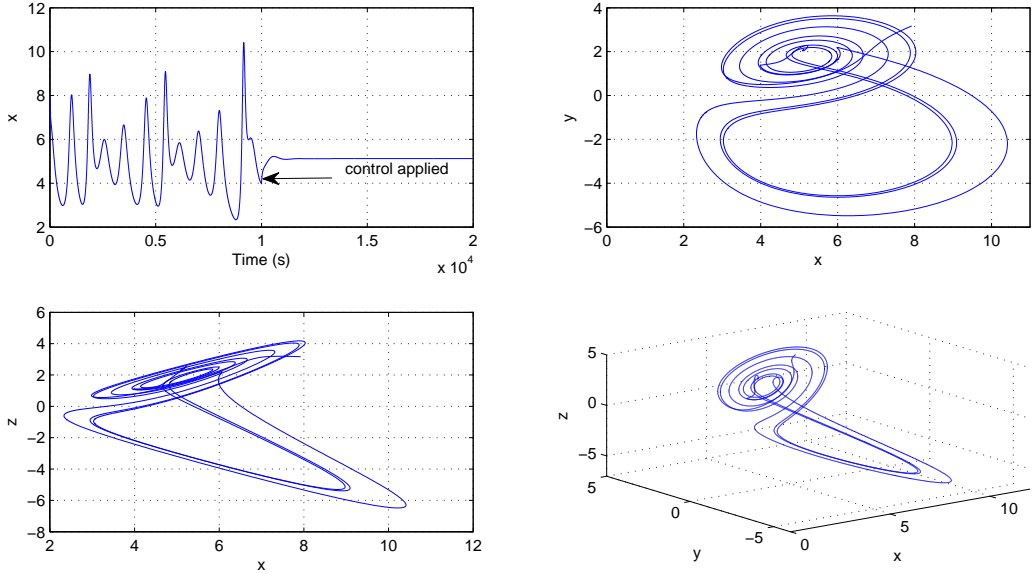}
\end{center}
\caption{Numerical results for the controlled system \eqref{eq:25}
with $a=3$, $b = 2.7$, $c = 4.7$, $d = 2$, $h = 9$, ${k_1} = 16.96$,
${k_2} = {k_3} = 0$, when the fractional order is $\alpha = 0.90$
and initial conditions are $\left({{x_0},{y_0},{z_0}}\right) = (5,2,2)$.}
\label{fig9}
\end{figure}

In order to get a faster response,
one can easily consider the other two gains ($k_2$ and $k_3$) in the control law.
However, this introduces difficulties
in the real implementation of the control system.
We also note that in designing the stabilizing controllers,
because we are utilizing the linearized version
of the nonlinear system around its equilibrium,
we should select the initial conditions near the corresponding equilibrium.


\section{Conclusions and Future Work}
\label{sec:6}

Chaotic fractional order systems have an inherent potential in mechatronic applications,
particularly in secure telecommunication systems where the main part of a transmitter-receiver
configuration is the synchronization between master and slave blocks.
In this article we analyzed the dynamical behavior of a novel fractional order chaotic system.
The chaotic system generalizes the recent integer order system introduced in \cite{[26]}.
The local stability of the equilibria, using the fractional
Routh--Hurwitz conditions, was studied.
Furthermore, using Matignon's stability criteria,
the system was shown to be chaotic in a wide parameter range,
and to have an interesting complex dynamical behavior that varies according
with the values of the parameters $a$, $b$, $c$, $d$, and $h$.
The very rich nonlinear dynamics include chaos and period doubling bifurcations.
Moreover, we derived a lower bound of the fractional order of differentiation for the system to remain chaotic.
Analytical conditions for linear feedback control have been achieved.
Our analysis is valid in a general format,
in which all gains $k_1$, $k_2$, and $k_3$ are considered.
However, and despite the complexity and the wide range of varieties,
it has been shown that the system can be controlled by a single
state-feedback controller. This possibility of stabilizing
the system locally using only controller $k_1$
provides a simple and easy way to control the chaos,
which can be crucial in a real implementation.
Our state feedback approach can be applied
to various chaotic fractional order systems. In particular,
we claim that the techniques here developed can be used,
without fundamental changes,
in the synchronization of two fractional order chaotic systems
in a secure telecommunication system.
This is under study and will be addressed elsewhere.

Our state feedback approach for controlling the system
is based on a Routh--Hurwitz analysis and cannot treat constraints
on the actuating signal. To consider such saturation constraints
on the control signal, a performance index should be defined. Minimizing
the index subject to some constraints on the state and control signals
cannot be considered analytically for the method here implemented.
A direction for future research is to investigate
how one can obtain a global stabilizing controller.
To the best of our knowledge, the global stability of fractional systems
is an interesting open question and available results reduce
to those of \cite{[R01],[R02]}. For different stability concepts
than the Lyapunov one adopted in our work, we refer the reader to \cite{[R03]}.


\section*{Acknowledgements}

This work was partially supported by FEDER through COMPETE,
Operational Programme Factors of Competitiveness, and by Portuguese funds through
the Center for Research and Development in Mathematics and Applications
(CIDMA, University of Aveiro) and the Portuguese Foundation for Science and Technology (FCT),
within project PEst-C/MAT/UI4106/2011 with COMPETE number FCOMP-01-0124-FEDER-022690.
The authors are grateful to two anonymous referees for valuable comments and helpful suggestions.



\end{document}